\documentclass[12pt]{amsart}
\usepackage{color}
\usepackage{bbm}
\usepackage{amsmath}
\usepackage{amsfonts}
\usepackage{mathrsfs}
\usepackage{amsthm}
\usepackage{txfonts}
\usepackage{amssymb}
\usepackage{color}
\usepackage{pict2e}
\usepackage{palatino,amssymb,epsfig,amsmath,latexsym,amsthm}
\usepackage{palatino,amssymb,epsfig,amsmath}
\usepackage{latexsym, amsthm, epsf, xy, epic, amscd}
\usepackage{palatino, amsmath, amsthm, amssymb, amscd, mathrsfs}

\theoremstyle{plain}
\newtheorem{theorem}{Theorem}[section]
\newtheorem{lemma}[theorem]{Lemma}

\newtheorem{proposition}[theorem]{Proposition}

\newtheorem{Bounded Diameter Lemma}[theorem]{Bounded Diameter Lemma}
\theoremstyle{definition}
\newtheorem{definition}[theorem]{Definition}

\newtheorem{remark}[theorem]{Remark}

\def\dist{{\text{dist}}}

\title{Tame the flexibility of circle patterns}

\author{Ze Zhou}
\address{Institute of Mathematics, Hunan University, Changsha 410082, China}
\email{zhouze@hnu.edu.cn}

\thanks{This work was supported by NSF of China (N0.11601141 and No.11631010)}
\date{}

\begin{document}


\begin{abstract}
This paper proves a deformation circle pattern theorem, which gives a complete description of those circle patterns with interstices in terms of the combinatorial type, the exterior intersections angles and the conformal structures of interstices. As results, the surface version of Rivin's theorem and the approximation property of packable surfaces are obtained.


\end{abstract}

\maketitle


\section{Introduction}
A circle pattern $\mathcal P$ on a Riemannian surface $(S,\mu)$ is a collection of closed disks. Particularly, if every pair of disks in $\mathcal P$ is either touched or disjoint, we often call $\mathcal P$ a circle packing. The contact graph $G(\mathcal P)$ of $\mathcal P$ is a graph having a vertex for each disk, and having an edge between the vertices $v,u$ for each connected component of $D_v\cap D_u$. Let $V_{\mathcal P}, E_{\mathcal P}$ denote the sets of vertices and edges of $G({\mathcal P})$. For an edge $e=[v,u]\in E_{\mathcal P}$ associated to a component $\mathcal E\subset D_v\cap D_u$, the exterior intersection angle $\Theta(e)$ is supplementary to the angle between the exterior normals of the boundaries $\partial D_v, \partial D_u$ at an intersection point in $\mathcal E$. In addition, for each connected component of the complement of the union of the interiors of the disks in $\mathcal P$, we call it an interstice. Please refer to Stephenson's monograph \cite{Stephenson1} for an exposition on circle patterns.

In this paper we will study the realization and characterization problem of circle patterns on hyperbolic surfaces. Precisely, suppose that $G$ is a graph embedding into a compact oriented surface $S$ of genus $g>1$ and $\Theta: E\mapsto[0,\,\pi)$ is a function defined in the edge set of $G$. Does there exist a circle pattern $\mathcal P$ whose contact graph is $G$ and whose exterior intersection angle function is $\Theta$? And if it does, how to characterize the solution space? For simplicity, we mainly focus on these circle patterns whose interstices correspond to the faces of the graph $G$ bijectively. The following consequence was due to Thurston \cite[Chap. 13]{Thurston}.

\begin{theorem}[Thurston] \label{T-1-1}
  Let  $G$ be the 1-skeleton of a triangulation $\mathcal{T}$ of a compact oriented surface $S$ of genus $g>1$. Assume that $\Theta: E\mapsto [0,\,\pi/2]$ is a function satisfying the conditions below:
 \begin{itemize}
 \item[$(i)$] If the edges $e_1,e_2,e_3$  form a null-homotopic closed path in $S$, then $\sum_{i=1}^3 \Theta(e_i)<\pi$;
\item[$(ii)$] If the edges $e_1,e_2,e_3,e_4$ form a null-homotopic closed path in $S$, then $\sum_{i=1}^4 \Theta(e_i)<2\pi$.
 \end{itemize}
 Then there exists a hyperbolic metric $\mu$ on $S$ such that $(S,\mu)$ supports a circle pattern $\mathcal P$ with the contact graph $G$ and the exterior intersection angles given by $\Theta$.
 Moreover, the pair $({\mu} ,\mathcal P)$ is unique up to isometry.
 \end{theorem}

 It is of interest to relax the requirement of non-obtuse angles in the above theorem. Recently, this was resolved by Zhou \cite{Zhou}. Before stating the result, let us explain some terminologies.

 For a circle pattern $\mathcal P$ on a Riemannian surface $(S,\mu)$, the \textbf{primitive contact graph} $G_{\bigtriangleup}(\mathcal P)$ of $\mathcal P$ is obtained from the contact graph $G(\mathcal P)$, by removing every edge whose intersection component is contained in a third closed disk in $\mathcal P$ (see FIGURE 1). Namely, an edge $e=[v,u]\in E_\mathcal P$ appears in $G_{\bigtriangleup}(\mathcal P)$ if and only if there exists no vertex $w\in V_\mathcal P\setminus\{v,u\}$  such that $\mathcal E\subset D_w$, where $\mathcal E\subset D_v\cap D_u$ denotes the intersection component for $e$.

\begin{figure}[htbp]\centering
\includegraphics[width=0.51\textwidth]{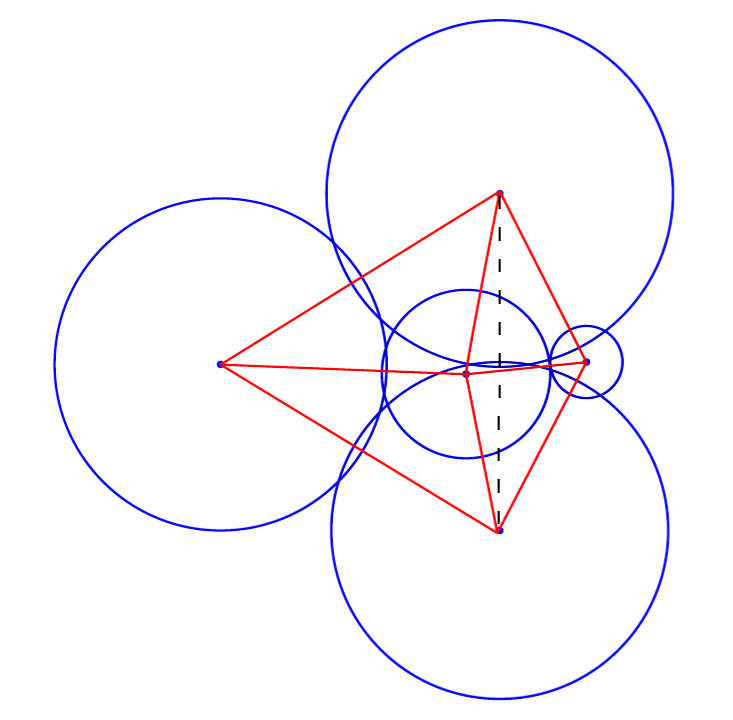}
\caption{The dashed edge is removed}
\end{figure}

A closed (not necessarily simple) path $\gamma$ in a compact oriented surface $S$ is called \textbf{pseudo-Jordan}, if $S\setminus \{\gamma\}$ has a simply-connected component with boundary $\gamma$.
Below is the result of Zhou \cite{Zhou}. We mention that a weaker version of this theorem can be derived from  Schenker's work \cite{Schl}.

 \begin{theorem}\label{T-1-2}
 Let $G$ be the 1-skeleton of a triangulation $\mathcal{T}$ of a compact oriented surface $S$ of genus $g>1$. Assume that $\Theta:E\mapsto [0,\,\pi)$ is a function satisfying
 \[
\sum\nolimits_{i=1}^s \Theta(e_i) \, <\, (s-2)\pi,
 \]
 whenever $e_1,e_2,\cdots,e_s$  form a \textbf{pseudo-Jordan} path in $S$. Then there exists a hyperbolic metric $\mu$ on $S$ such that $(S,\mu)$  supports a circle pattern $\mathcal P$ with the {\it\textbf{primitive contact graph}} $G$ and the exterior intersection angles given by $\Theta$. Moreover, the pair $(\mu,\mathcal P)$ is unique up to isometry.
\end{theorem}

One may ask the following question: what if $G$ is generally the 1-skeleton of a cellular decompositions $\mathscr{D}$ of $S$? Considering $G$ as a subgraph of certain triangular graph, the existence follows from Theorem \ref{T-1-2}. Nonetheless, the rigidity may not hold any more. For instance, under the assumption that some $2$-cells of $\mathscr{D}$ are quadrangles, Brooks' results \cite{BK1,BK2} imply that up to isometry there exist uncountable many pairs $(\mu, \mathcal P)$ realizing the same data. Now a natural problem arises: how to describe the moduli space of solutions?


 The first progress was made by Brooks \cite{BK1,BK2}, who gave a description in case that all $2$-cells of $\mathscr{D}$ were either triangles or quadrangles by a kind of continued fraction parameters. For those circle patterns on the Riemann sphere with non-obtuse exterior intersection angles, He-Liu \cite{HL} developed a deformation theory showing that the solution space is identified with the product of the Teichm\"{u}ller spaces of  interstices. See also Huang-Liu \cite{Hu-L} for a related result on convex hyperbolic polyhedra. Aside from these special cases, the problem remains open. In this paper we shall give a complete description for circle patterns on higher genus surfaces.

 Let us introduce a notion as an analogue of the quasiconformal quadrangle. For basic backgrounds on quasiconformal mappings and Teichm\"{u}ller theory, the readers can refer to \cite{Ahl,Ima,LV}. Given a topological polygonal region $I\subseteq \mathbb{D}$ in the hyperbolic disk, we consider all quasiconformal embeddings $h:I\mapsto\mathbb{D}$. Say two such embeddings $h, \tilde{h}: I\mapsto
\mathbb{D}$ are Teichm\"{u}ller equivalent, if the composition
mapping
\[
h\circ \tilde{h}^{-1}: h(I)\mapsto \tilde{h}(I)
\]
 is homotopic to a conformal homeomorphism $\varphi$ which maps
$h(e_{i})$ onto $\tilde{h}(e_{i})$ for each side $e_{i}\subset\partial  I$ of $I$. In other words, there exists a mark-preserving conformal mapping between their images. Here a homeomorphism between two polygonal region is called mark-preserving if the $i$-th side of one region is mapped into the $i$-th side of the other.

 \begin{definition}
  The Teichm\"{u}ller space $T(I)$ of $I$ is the set of  equivalence classes of quasiconformal embeddings $h: I\mapsto \mathbb{D}$.
 \end{definition}

\begin{remark}\label{R-2-5}
The space $T(I)$ parameterizes all the conformal structures on $I$. Moreover, if $I$ is $m$-sided, then $T(I)$ is diffeomorphic to the Euclidean space $\mathbb{R}^{m-3}$. See e.g. \cite{LV}.
\end{remark}

Let $E,F$ be the sets of edges and $2$-cells of $\mathscr{D}$. By Theorem \ref{T-1-1}, there exists at least one hyperbolic metric $\mu^0$ on $S$ such that $(S,\mu^0)$ supports a circle packing $\mathcal P^0$ with contact graph $G$. Suppose that $\big\{\,I_1^0,\cdots,I_{|F|}^0\,\big\}$ are all the interstices of $\mathcal P^0$. Note that each $ I_\alpha^0$ is a polygonal region. We define
\[
T(G(\mathcal P^0))\,=\,\prod\nolimits_{\alpha=1}^{|F|}T(I_\alpha^0).
\]
Due to Remark \ref{R-2-5}, it follows that
\[
T(G(\mathcal P^0))\,\cong\, \mathbb{R}^{2|E|-3|F|}.
\]
Here what the choice $\mathcal P^0$ is to $T(G(\mathcal P^0))$,  that a base surface $S_0$ is to the Teichm\"{u}ller space of compact surfaces. Since different choices give the same topological information, we simply denote it by $T(G)$.

\begin{theorem}\label{T-1-5}
 Let $G$ be the 1-skeleton of a cellular decomposition $\mathscr{D}$ of a compact oriented surface $S$ of genus $g>1$. Assume that $\Theta: E\mapsto [0,\pi)$ is a function such that
\[
\sum\nolimits_{i=1}^s\Theta(e_i)<(s-2)\pi,
\]
whenever the edges $e_1,e_2,\cdots,e_s$  form a \textbf{pseudo-Jordan} path in $S$. Suppose that $T(G)$ is defined as above. For any
\[
\big[h(G)\big]\,=\,\big(\,h_1,h_2,\cdots,h_{|F|}\,\big)\in T(G),
\]
there exists a hyperbolic metric $\mu$ on $S$ such that $(S,\mu)$ supports a circle pattern $\mathcal P$ with the {\it\textbf{primitive contact graph}} $G$, the exterior intersection angles given by $\Theta$, and the interstices equipped with  conformal structures assigned by $[h(G)]$. Moreover, the pair $(\mu,\mathcal P)$ is unique up to isometry.
\end{theorem}

\begin{remark}
Note that the Teichm\"{u}ller space of a triangle consists of a point. Thus Theorem \ref{T-1-2} is a corollary of the above result.
\end{remark}

 A circle pattern $\mathcal P$ on $(S,\mu)$ is called \textbf{ideal} if it satisfies the following properties: $(i)$ $\mathcal P$ has the \textbf{primitive contact graph} isomorphic to the 1-skeleton of a cellular decomposition of $S$ whose $2$-cells correspond to the interstices of $\mathcal P$ bijectively; $(ii)$ every interstice of $\mathcal P$ consists of a point. By Theorem \ref{T-1-5}, we will show the following result obtained by Bobenko-Springborn \cite{Boben}. It is closely related to Rivin's theorem on ideal hyperbolic polyhedra \cite{Rivin1}.
\begin{theorem}\label{T-1-7}
 Let $G$ be the 1-skeleton of a cellular decomposition $\mathscr{D}$ of a compact oriented surface $S$ of genus $g>1$. Assume that $\Theta: E\mapsto (0,\,\pi)$ is a function such that the following conditions hold:
 \begin{itemize}
\item[\textbf{(H1)}]When the distinct edges $e_1,e_2,\cdots,e_m$ form the boundary of a $2$-cell of $\mathscr{D}$, $\sum_{l=1}^m \Theta(e_l)=(m-2)\pi$.
\item[\textbf{(H2)}]When the edges $e_1,e_2,\cdots,e_s$ form a \textbf{pseudo-Jordan} path which is not the boundary of a 2-cell of $\mathscr{D}$, $\sum_{l=1}^s \Theta(e_l)<(s-2)\pi$.
\end{itemize}
Then there exists a hyperbolic metric $\mu$  on $S$ such that $(S,\mu)$  supports an\textbf{ ideal}  circle pattern  $\mathcal P$  with the {\it\textbf{primitive contact graph}} $G$ and the exterior intersection angles  given by $\Theta$. Moreover, the pair $(\mu,\mathcal P)$ is unique up to isometry.
\end{theorem}

A compact hyperbolic surface is called packable if it supports a circle packing whose contact graph is triangular. Applying Theorem \ref{T-1-5}, we will prove a result due to Brooks \cite{BK2}, which asserts that any compact hyperbolic surface can be approximated by packable ones. See also the work of Bowers-Stephenson \cite{BS3}.

\begin{theorem}[Brooks]\label{T-1-8}
 The packable surfaces form a density set in the Teichm\"{u}ller space of $S$.
 \end{theorem}

The paper is organized as follows. In Section 2, we collect several preparatory results concerning three-circle and multi-circle configurations. Section 3 gives a sketch proof of Theorem \ref{T-1-2}.  In Section 4, we present the proofs of the main results of this paper. Finally, there are three appendixes on univalent functions and multi-circle configurations.

\section{Preliminaries}
\subsection{Three-circle configurations } To prove Circle Pattern theorem, Thurston ever formulated several lemmas on configurations of three circles meeting in non-obtuse angles \cite{Thurston}. Below is a similar result proved by Zhou \cite{Zhou}.
\begin{lemma}\label{L-2-1}
 Given $r_i, r_j, r_k>0$ and three angles $\Theta_i, \Theta_j, \Theta_k\in[0,\pi)$ satisfying
\[
\Theta_i+\Theta_j+\Theta_k\,<\,\pi,
\]
 there exists a configuration of three mutually intersecting closed disks in hyperbolic geometry, unique up to isometry, having radii $r_i, r_j, r_k$  and meeting in exterior intersection angles $\Theta_i, \Theta_j, \Theta_k$.
\end{lemma}

As FIGURE 2, let $\vartheta_i,\vartheta_j,\vartheta_k$ be the inner angles of the triangle of centers of the three disks. Lemma \ref{L-2-2} is a combination of Zhou \cite{Zhou} and Xu \cite{Xu}. See also Chow-Luo \cite{Luo1}, Guo \cite{Guo} for some related results.

\begin{figure}[htbp]\centering
\includegraphics[width=0.75\textwidth]{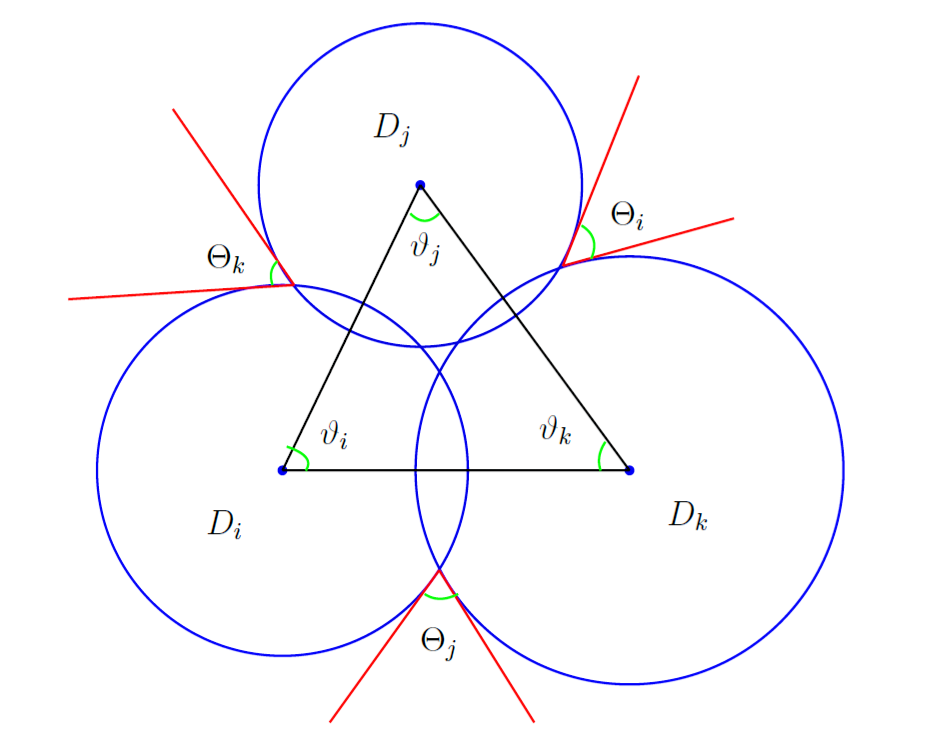}
\caption{A three-circle configuration}
\end{figure}

\begin{lemma}\label{L-2-2}
 Suppose that $\Theta_i, \Theta_j, \Theta_k$  satisfy the conditions of Lemma \ref{L-2-1}. Then
\begin{itemize}
\item[$\langle a\rangle$]$\displaystyle{
\sinh r_j\dfrac{\partial\vartheta_i}{\partial r_j}\,=\,\sinh r_i\dfrac{\partial\vartheta_j}{\partial r_i}};
$
\item[$\langle b\rangle$]$\displaystyle{
\dfrac{\partial\vartheta_i}{\partial r_i}\ <\ 0,\quad \dfrac{\partial\vartheta_j}{\partial r_i}\ > \ 0,\,\quad \dfrac{\partial(\vartheta_i+\vartheta_j+\vartheta_k)}{\partial r_i}\ <\ 0}.
$
\end{itemize}
\end{lemma}


\begin{lemma}\label{L-2-3}
Under the conditions of Lemma \ref{L-2-1}, let $r=(r_i,r_j,r_k)$ and $\tilde{r}=(\tilde{r}_i,\tilde{r}_j,\tilde{r}_k)$ be two vectors of positive numbers such that
\[
\max_{\tau=i,j,k}\Bigg\{\,\frac{\tanh(r_\tau/2)}{\tanh(\tilde r_\tau/2)}\,\Bigg\}
\,=\,\frac{\tanh(r_i/2)}{\tanh(\tilde r_i/2)}\,>\,1.
\]
Then
\[
\vartheta_i(r)\,<\,\vartheta_i(\tilde{r}).
\]
\end{lemma}
\begin{proof}
For $\tau=i,j,k$, let $u_\tau=\log\tanh(r_\tau/2)$. Then $\vartheta_i,\vartheta_j,\vartheta_k$ are smooth functions of $u=(u_i,u_j,u_k)$. Due to Lemma \ref{L-2-2}, we obtain
\[
\dfrac{\partial\vartheta_i}{\partial u_j}\,=\,\dfrac{\partial\vartheta_j}{\partial u_i}\,>\,0
\]
and
\[
\frac{\partial(\vartheta_i+\vartheta_j+\vartheta_k)}{\partial u_i}\,<\,0.
\]
For $t\in[0,1]$, let $r_\tau(t)>0$ such that
\[
tu_\tau+(1-t)\tilde{u}_\tau\,=\,\log\tanh(r_\tau(t)/2).
\]
Define
\[
F(t)=\vartheta_i\big(r(t)\big).
 \]
 By the mean value theorem, there exists $\xi\in(0,1)$ such that
\[
\vartheta_i(r)-\vartheta_i(\tilde{r})\,=\,F(1)-F(0)\,=\,F'(\xi).
\]
Setting $u(\xi)=\xi u+(1-\xi)\tilde{u}$, from Lemma \ref{L-2-2}, it follows that
\[
\begin{aligned}
F'(\xi)\,=\, &\, \big(\,u_i-\tilde{u}_i\,\big)\frac{\partial\vartheta_i}{\partial u_i}(u(\xi))+\big(\,u_j-\tilde{u}_j\,\big)\frac{\partial\vartheta_i}{\partial u_j}(u(\xi))+\big(\,u_k-\tilde{u}_k\,\big)\frac{\partial\vartheta_i}{\partial u_k}(u(\xi))\\
\leq\, &\, \max_{\tau=i,j,k}\big\{(\,u_\tau-\tilde{u}_\tau\,\big)\}\Big(\,\frac{\partial\vartheta_i}{\partial u_i}(u(\xi))+\frac{\partial\vartheta_i}{\partial u_j}(u(\xi))+\frac{\partial\vartheta_i}{\partial u_k}(u(\xi))\,\Big)\\
=\, &\, \big(\,u_i-\tilde{u}_i\,\big)\frac{\partial(\vartheta_i+\vartheta_j+\vartheta_k)}{\partial u_i}(u(\xi))\\
<\, &\, 0.
\end{aligned}
\]
Consequently,
\[
\vartheta_i(r)\,<\,\vartheta_i(\tilde{r}).
\]
\end{proof}

\begin{lemma}\label{L-2-4}
Under the conditions of Lemma \ref{L-2-1}, we have
\[
\displaystyle{\lim_{r_i\to\infty}\vartheta_i\,=\,0},
\]
and
\[
\begin{aligned}
&\lim_{(r_i,r_j,r_k)\to(0,a,b)}\vartheta_i\,=\,\pi-\Theta_i, \\
&\lim_{(r_i,r_j,r_k)\to(0,0,c)}\vartheta_i+\vartheta_j\,=\,\pi,  \\
&\lim_{(r_i,r_j,r_k)\to(0,0,0)}\vartheta_i+\vartheta_j+\vartheta_k\,=\,\pi, \\
\end{aligned}
\]
where $a,b,c$ are fixed positive numbers.
\end{lemma}

The proof is based on routine calculations. For details, please refer to Zhou \cite{Zhou}. See also Ge-Jiang \cite{Ge-Jiang} and Ge-Xu \cite{Ge-Xu}.

\subsection{Degenerating cases}
For a three-circle configuration, as one of the radii $r_k\to 0$, it will degenerate to a pair of intersecting disks\,(see FIGURE 3). Below are several parallel results.

\begin{lemma}\label{L-2-5}
For any $r_i, r_j>0$ and $\Theta_k\in(0,\pi)$,
 there exists a configuration of two mutually intersecting closed disks in hyperbolic geometry, unique up to isometry, having radii $r_i, r_j$  and meeting in exterior intersection angle $\Theta_k$.
\end{lemma}

\begin{figure}[htbp]\centering
\includegraphics[width=0.69\textwidth]{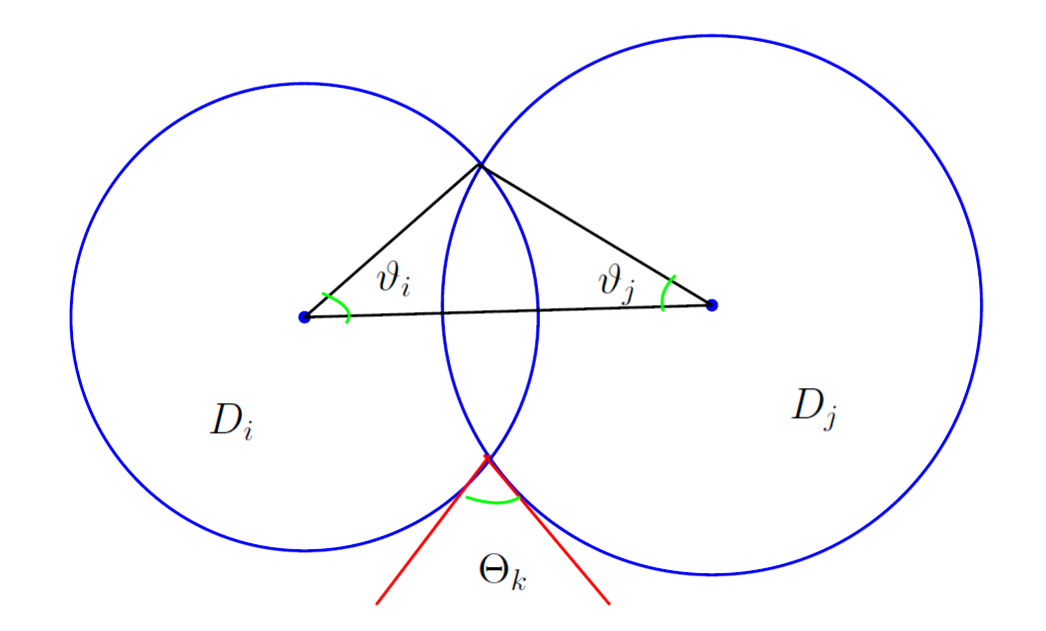}
\caption{}
\end{figure}

Fixing  $\Theta_k\in (0,\,\pi)$, then the inner angles $\vartheta_i,\vartheta_j$ are smooth functions of  $r=(r_i,r_j)$.
\begin{lemma}\label{L-2-6}
We have
\begin{itemize}
\item[$\langle a\rangle$]
$\displaystyle{
\sinh r_j\dfrac{\partial\vartheta_i}{\partial r_j}\,=\,\sinh r_i\dfrac{\partial\vartheta_j}{\partial r_i}}$;
\item[$\langle b\rangle$]
$\displaystyle{
\dfrac{\partial\vartheta_i}{\partial r_i}\ <\ 0,\quad \dfrac{\partial\vartheta_j}{\partial r_i}\ > \ 0,\,\quad \dfrac{\partial(\vartheta_i+\vartheta_j)}{\partial r_i}\ <\ 0}.$
\end{itemize}
\end{lemma}

\begin{lemma}\label{L-2-7}
Given two radius vectors $r=(r_i,r_j)$ and $\tilde{r}=(\tilde{r}_i,\tilde{r}_j)$ satisfying
\[
\max_{\tau=i,j}\Bigg\{\frac{\tanh(r_\tau/2)}{\tanh(\tilde r_\tau/2)}\Bigg\}
\,=\,\frac{\tanh(r_i/2)}{\tanh(\tilde r_i/2)}\,>\,1,
\]
then
\[
\vartheta_i(r)\,<\,\vartheta_i(\tilde{r}).
\]
\end{lemma}

\begin{lemma}\label{L-2-8}
Under the conditions of Lemma \ref{L-2-5}, then
\[
\displaystyle{\lim_{r_i\to\infty}\vartheta_i\,=\,0},
\]
and
\[
\begin{aligned}
&\lim_{(r_i,r_j)\to(0,d)}\vartheta_i\,=\,\Theta_k, \\
&\lim_{(r_i,r_j)\to(0,0)}\vartheta_i+\vartheta_j\,=\,\Theta_k,  \\
\end{aligned}
\]
where $d$ is a fixed positive number.
\end{lemma}

\subsection{Multi-circle configurations}

A $m$-circle ($m\geq 3$) configuration $\mathcal P$ is the union of a $m$-sided interstice and its adjacent disks. Namely, it is a circle pattern with the \textbf{primitive contact graph}  isomorphic to a $m$-sided polygon.

\begin{lemma}[Rigidity]\label{L-2-9}
 Let $\mathcal P,\,\tilde{\mathcal P}$ be two $m$-circle configurations in hyperbolic geometry, with the same radius vector and exterior intersection angle function. If their interstices share the same conformal structure, then $\mathcal P,\tilde{\mathcal P}$ are isometric.
\end{lemma}

\begin{figure}[htbp]\centering
\includegraphics[width=0.72\textwidth]{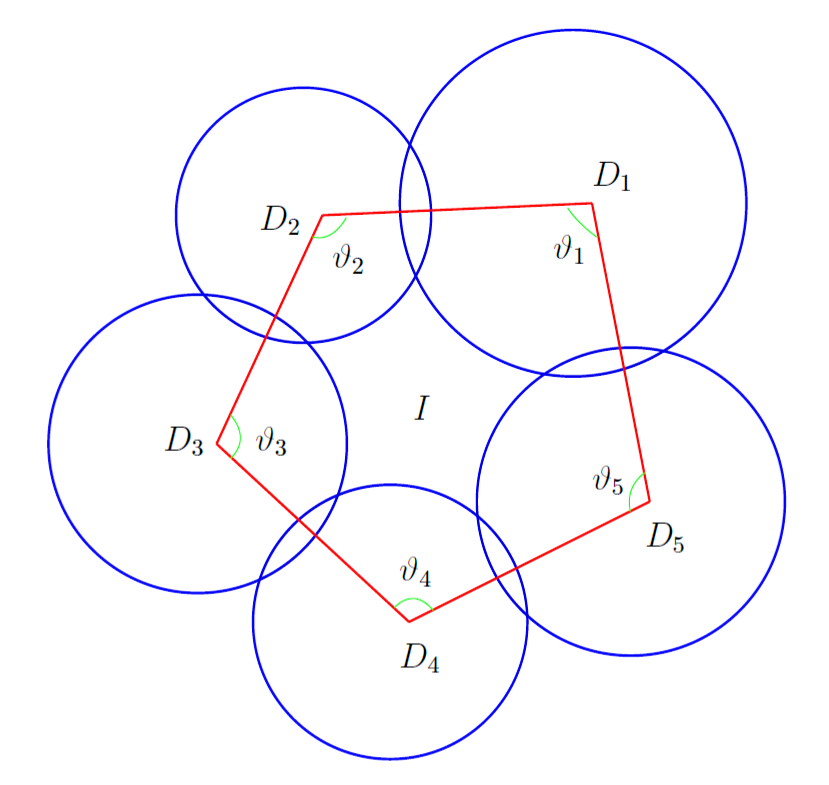}
\caption{A five-circle configuration}
\end{figure}

Let $\vartheta_1,\cdots,\vartheta_m$ be the inner angles of the  polygon of centers of the disks of $\mathcal P$. Fixing the exterior intersection angles and the conformal structure of the interstice, Lemma \ref{L-2-9} implies that $\vartheta_1,\cdots,\vartheta_m$ are well-defined functions of the radius vector $r=(r_1,\cdots,r_m)$.

\begin{lemma}[Maximal principle]\label{L-2-10}
  Suppose that $\mathcal P,\,\tilde{\mathcal P}$ are two $m$-circle configurations in hyperbolic geometry whose interstices are equipped with the same conformal structure. If their exterior intersection angle functions are the same and their radius vectors $r=(r_1,\cdots, r_m)$ and $r=(\tilde{r}_1,\cdots,\tilde{r}_m)$ satisfy
\[
\max_{1\leq \tau\leq m}\Bigg\{\frac{\tanh(r_\tau/2))}{\tanh(\tilde r_\tau/2}\Bigg\}
\,=\,\frac{\tanh(r_i/2)}{\tanh(\tilde r_i/2)}\,>\,1,
\]
then
\[
\vartheta_i(r)\,<\,\vartheta_i(\tilde{r}).
\]
\end{lemma}

\begin{lemma}\label{L-2-11}
 $\vartheta_i\rightarrow 0$ , as $r_i\to \infty$.
 \end{lemma}

The proofs of the above three lemmas are involved and based on some delicate estimates on univalent functions. For details, please refer to Appendix 2 and Appendix 3.

\begin{lemma}\label{L-2-12}
Given a family of  $m$-circle configurations $\{\mathcal P_r \}$  with the interstices $\{I_r\}$ sharing the same conformal structure, let $r=(r_1,\cdots,r_m)$ be the radius vector of $\mathcal P_r$. If $r_i\to 0$ for some $i\in\{1,\cdots,m\}$, then $I_r$ degenerates to a point.
\end{lemma}

\begin{proof}
  Without loss of generality, suppose that the $i$-th circle $C_{r,i}$ of $\mathcal P_r$ centers at the origin. Choose $\mathcal P\in \{\mathcal P_r\}$. Let $I$ and $C_i$ denote the interstice and $i$-th circle of $\mathcal P$. Because $\{I_r\}$ share the same conformal structure, there exists a mark-preserving conformal map $\varphi_r: I\mapsto I_r$. By Schwarz's reflection principle, we can extend it to $\check{\varphi}_r: \check{I}\mapsto \check{I}_r$ by reflections cross $C_i$ and $C_{r,i}$.

 For holomorphic function family $\{\check{\varphi}_{r}\}$, using Montel's theorem, we extract subsequence $\{\check{\varphi}_{r_k}\}$ such that
\[
\check{\varphi}_{r_k}\,\to\, \varphi_\ast\qquad \text{or} \qquad \bar{\varphi}_{r_k}\,\to\, \text{const},
\]
where  $\varphi_\ast$ is a non-constant holomorphic function.

The first case does not occur. Otherwise, pick $\mathbb{D}(z_0,\varrho)\subset \mathbb D_i\cap \check{I}$. Then the images $\check{\varphi}_r\big(\mathbb{D}(z_0,\varrho)\big)$ is contained in $\mathbb{D}_{r,i}\cap\check{I}_r$. Here $\mathbb D_i$ and $\mathbb{D}_{r,i}$ are the open disks bounded by $C_i$ and $C_{r,i}$. As $r_i\to 0$, it follows from Cauchy's integral formula that
\[
|\check{\varphi}'_{r}(z_0)|\,\to\, 0.
\]
Consequently,
\[
\varphi_\ast'(z_0)\,=\,0.
\]
On the other hand, because each $\check{\varphi}_r$ is one-to-one, Rouch\'{e}'s theorem implies that  $\varphi_\ast$ is also one-to-one. Hence the Jacobian is positive. Namely,
\[
|\varphi_\ast'(z)|\,>\,0,\;\forall\; z\in \check{I}.
\]
This leads to a contradiction.

 As a result, $\check{\varphi}_{r_k}$  tends to a constant. Moreover, similar argument implies that any other convergent subsequence of $\{\check{\varphi}_r\}$ tends to the same constant. Thus we show that the interstice $ I_r$ degenerates to a  point.
\end{proof}

A $m$-circle configuration is called ideal if its interstice consists of only one point.
\begin{lemma}\label{L-2-13}
Lemma \ref{L-2-9}, Lemma \ref{L-2-10} and Lemma \ref{L-2-11} hold for ideal $m$-circle configurations.
 \end{lemma}

The proof is left to Appendix 3.

\section{A sketch proof of Theorem \ref{T-1-2}}
Theorem \ref{T-1-2} plays a crucial role in the proof of our main result. For completeness,  in this section we give a sketch proof of this theorem.

\subsection{Thurston's construction} Recall that $S$ is a compact oriented surface of genus $g>1$ and $\mathcal{T}$ is a triangulation of $S$ with the sets of vertices, edges and triangles $V, E, F$. Let $r\in \mathbb R_{+}^{|V|}$ be a radius vector, which assigns each vertex $v\in V$ a positive number $r(v)$. Then $r$ together with the weight  function $\Theta: E\mapsto [0,\,\pi)$ in Theorem \ref{T-1-2}  determines a hyperbolic cone metric structure on $S$ as follows.

For each triangle $\triangle(v_i,v_j,v_k)$ of $\mathcal{T}$ with vertices $v_i,v_j,v_k$,  we associate the triangle determined by the centers of three mutually intersecting disks with hyperbolic radii
\[
 r(v_i),\,r(v_j),\,r(v_k)\,
 \]
 and exterior intersection angles
 \[
 \Theta\big([v_i,v_j]\big),\,\Theta\big([v_j,v_k]\big),\,\Theta\big([v_k,v_i]\big).
 \]
By Lemma \ref{L-2-1}, the above procedure works well.

Gluing all of these hyperbolic triangles produces a metric surface $(S,\mu(r))$ which is locally hyperbolic with possible cone type singularities at the vertices. For each $v\in V$, the discrete curvature $k(v)$ is defined as
\[
k(v)\;=\;2\pi\ -\ \sigma(v),
\]
where $\sigma(v)$  denotes the cone angle at $v$. More precisely, $\sigma(v)$ is equal to the sum of all inner angles having vertex $v$.

Write $k(v)$ as $k(v)(r)$, then $k(v)$ is a smooth function $r$. If there exists a radius vector $r^0$ such that $k(v)(r^0)=0$ for all $v\in V$, then the cone singularities turn into smooth. For every $v\in V$, on $(S,\mu(r^0))$ drawing the disk centering at $v$ with radius $r^{0}(v)$, we will obtain a circle pattern realizing $(G,\Theta)$. Consider the following curvature map
\[
\begin{aligned}
Th:\quad\qquad &\,\mathbb{R}_{+}^{|V|} \quad
&\,\longmapsto \,\;\qquad\qquad &\,\mathbb{R}^{|V|} \\
\big(\,r(v_1),\, &r(v_2),\,\cdots\,\big)&\longmapsto\;\quad\big(\,k(v_1),\, &k(v_2),\,\cdots\,\big).\\
  \end{aligned}
  \]
It remains to prove that the origin $\emph{O}=(0,\cdots,0)\in \mathbb R^{|V|}$ belongs to the image of the map $Th$.

\subsection{Continuity method} To achieve the goal, a key step is to give a description of the image set. Let $Y\subset \mathbb{R}^{|V|}$ be the convex set characterized by the following system of inequalities
\begin{equation}\label{E-6}
k(v)\ <\ 2\pi, \; \forall v\,\in V,
\end{equation}
and
\begin{equation}\label{E-7}
\sum_{v \in V_0}k(v) >\sum_{(e,v)\in Lk(V_0)}\big(\Theta(e)-\pi\big)+2\pi\chi\big(CK(V_0)\big)
\end{equation}
for any non-empty subset $V_0$ of $V$.
Here $CK(V_0)$, with the Euler characteristic $\chi\big(CK(V_0)\big)$, denotes the union of these cells of $\mathcal T$ having at least one vertex in $V_0$, and $Lk(V_0)$ denotes the set of pairs $(e, v)$ of an edge $e$ and a vertex $v$ satisfying: $(i)$ the end vertices of $e$ are not in $V_0$; $(ii)$ $v$ is in $V_0$; $(iii)$ $e$ and $v$ form a triangle of $\mathcal{T}$. We mention that the cells of $\mathcal{T}$ include the vertices ($0$-cells), the open edges ($1$-cells) and the open triangles ($2$-cells) of $\mathcal{T}$.

\begin{proposition}\label{P-2-4}
$Th(\mathbb{R}_{+}^{|V|})=Y$.
\end{proposition}

\begin{proof}
We have the following claims:
\begin{itemize}
\item[$\langle a\rangle$] $Th$ is continuous.
\item[$\langle b\rangle$] $Th$ is injective. Suppose that $r,\tilde{r}$ are two radius vectors giving the same discrete curvatures. We need to show $r=\tilde{r}$. Otherwise, without loss of generality, assume that $v_0\in V$ satisfies
    \[
\max_{v\in V}\Bigg\{\frac{\tanh(r(v)/2)}{\tanh(\tilde r(v)/2)}\Bigg\}
\,=\,\frac{\tanh(r(v_0)/2)}{\tanh(\tilde r(v_0)/2)}\,>\,1.
\]
From Lemma \ref{L-2-3}, it follows that
    \[
  k(v_0)(r)\,>\,k(v_0)(\tilde{r}),
    \]
    which leads to a contradiction.
\item[$\langle c\rangle$] $Th$ is proper. It suffices to check that one of the inequalities in (\ref{E-6}) and (\ref{E-7}) turns into equal as the radius vector goes beyond all compact sets of $\mathbb{R}_{+}^{|V|}$, which follows from Lemma \ref{L-2-4}. See Thurston's work \cite[Chap. 13]{Thurston} for details.
\end{itemize}

Due to Brouwer's theorem on invariance of domain, the first two clams imply that $Th(\mathbb{R}_{+}^{|V|})$ is an open set of $Y$. On the other hand, it follows from the third claim that $Th(\mathbb{R}_{+}^{|V|})$ is closed in $Y$. Because $Y$ is connected, we have $Th(\mathbb{R}_{+}^{|V|})=Y$.
\end{proof}

 A circle pattern $\mathcal P$ on a hyperbolic surface $(S,\mu)$ is called $G$-type, if there exists a geodesic triangulation $\mathcal{T}(\mu)$ of $(S,\mu)$ with the following properties: $(i)$ $\mathcal{T}(\mu)$ is isotopic to $\mathcal{T}$; $(ii)$ the vertices of $\mathcal{T}(\mu)$ coincide with the centers of the disks in $\mathcal P$.
\begin{proof}[\textbf{Proof of Theorem \ref{T-1-2}}]
 First we claim that there exists a hyperbolic metric $\mu$ on $S$ such that $(S,\mu)$  supports a $G$-type circle pattern $\mathcal P$ with the exterior intersection angles given by $\Theta$. By Proposition \ref{P-2-4}, it suffices to show that $\emph{O}=(0,0,\cdots,0)\in Y$. More precisely, we need to check that
\begin{equation}\label{E-8}
\sum_{(e,v)\in Lk(V_0)}\big(\Theta(e)-\pi\big)+2\pi\chi\big(CK(V_0)\big)\,<\,0,\; \forall\,V_0\subset V,\,  V_0\neq\emptyset.
\end{equation}

Without loss of generality, suppose that $CK(V_0)$ is connected and is homotopic to a surface of topological type $(g_0,l_0)$, where $g_0,l_0$ denotes the genus and the number of boundary components, respectively. Thus we have
\[
\chi(CK(V_0))\, =\, 2-2g_0-l_0.
\]

If $g_0\geq 1$ or $l_0\geq 2$, then $\chi(CK(V_0))\leq 0$,  (\ref{E-8}) naturally holds.

Let $g_0=0$, $l_0=1$. Then $CK(V_0)$ is simply-connected. And
\[
\chi(CK(V_0))\, =\, 2-2g_0-l_0\,=\,1.
\]
Assume that  $e_1,\cdots,e_s$ are the edges marked with vertices $v_1,\cdots,v_s$ such that $(e_i,v_i)\in Lk(V_0)$ for $i=1,\cdots, s$. Note that $e_1,\cdots,e_s$ form a \textbf{pseudo-Jordan} path bounding $CK(V_0)$. According to the condition, we have
\[
\sum\nolimits_{i=1}^s \Theta(e_i)\,<\,(s-2)\pi.
\]
It follows that
\[
\sum_{(e,f)\in Lk(V_0)}\big(\Theta(e)-\pi\big)+2\pi\chi(CK(V_0))\,=\,\sum_{i=1}^s \big(\Theta(e_i)-\pi\big)+2\pi\,<\,0.
\]

Meanwhile, the injectivity of $Th$  implies that $(\mu,\mathcal P)$ is unique up to isometry. The final step is to check that the above $G$-type circle pattern $\mathcal P$  actually realizes $G$ as the \textbf{primitive contact graph}, which has been prove by Zhou \cite{Zhou}.
\end{proof}

\section{Proofs of main results}

\subsection{Deformation circle pattern theorem} The major purpose of this paper is to prove Theorem \ref{T-1-5}. Our strategy is similar to He-Liu \cite{HL} and Liu-Zhou \cite{L-Z}. To be specific, it is a combination of Theorem \ref{T-1-2} and Rodin-Sullivan's trick \cite{RS}.

\begin{proof}[\textbf{Proof of Theorem \ref{T-1-5}}]
The proof is divided into two parts.

\textbf{\textsl{Existence part.}} For each $ I_\alpha^0\in\{\,I^0_1,\cdots,I_{|F|}^0\,\}$, let $h_\alpha: I_\alpha^0 \mapsto\mathbb{D}$ be the preassigned conformal structure. Regard the image region $h_\alpha(I_\alpha^0)$ as a bounded domain in the complex plane $\mathbb{C}$. Lay down the regular hexagonal circle packing  in $\mathbb{C}$, with each circle of radius $1/n$. By using the boundary component $h_\alpha(I_\alpha^0)$ like a cookie-cutter, we obtain a circle packing $\mathcal Q_{n,\alpha}$ which consists of all disks intersecting with the region $h_\alpha(I_\alpha^0)$. Let $K_{n,\alpha}$ denote the contact graph of $\mathcal {Q}_{n,\alpha}$.

Jointing these graphs $\{K_{n,\alpha}\}_{\alpha=1}^{|F|}$ to the original graph $G$ along the corresponding boundaries produces a triangular graph $G_n$. Let $E_n$ be the edge set of $G_n$. Define a weight function $\Theta_n: E_n\mapsto [0,\pi)$ by setting $\Theta_n(e)=\Theta(e)$, if $e$ belongs to $E$. Otherwise, set $\Theta_n(e)=0$. Note that $(G_n,\Theta_n)$ satisfies the conditions of Theorem \ref{T-1-2}. It follows that there exists a hyperbolic metric $\mu_n$ on $S$ such that $(S,\mu_n)$ supports a circle pattern with primitive contact graph $G_n$ and the exterior intersection angles given by $\Theta_n$. Discarding those disks associated to those new vertices, we obtain a circle pattern $\mathcal P_n$ realizing $(G,\Theta)$.

Due to Lemma \ref{L-2-4}, there exists an upper bound for the radius of every disk in $\mathcal P_n$ as $n$ varies. Let us extract a subsequence of $\{(\mu_n,\mathcal P_n)\}$ convergent to a pre circle pattern pair $(\mu_\infty,\mathcal P_\infty)$. For simplicity, we still denote it by $\{(\mu_n,\mathcal P_n)\}$.

By the following Proposition \ref{P-4-1}, $\mathcal P_n$ does not degenerate. Hence $(\mu_\infty,\mathcal P_\infty)$ is a true circle pattern pair realizing $(G,\Theta)$.  Rodin-Sullivan's consequence \cite{RS} implies there exists a mark-preserving conformal mapping between the $\alpha$-th interstice of $(\mu_\infty,\mathcal P_\infty)$ and $h_\alpha(I_\alpha^0)$. Namely, the interstices are equipped with the conformal structures asigned by $[h(G)]$. Setting $(\mu,\mathcal P)=(\mu_\infty,\mathcal P_\infty)$ then gives the desired circle pattern pair.

\textbf{ \textsl{Rigidity part.}} Suppose that there exist
two circle pattern pairs $(\mu,\mathcal P)$ and $(\tilde \mu,\tilde{\mathcal P})$ with the desired properties. Let $r, \tilde{r}$ be the radius vectors of $\mathcal P$ and  $\tilde{\mathcal P}$, respectively. We claim that $r=\tilde{r}$.

Otherwise, without loss of generality, assume that $v_0\in V$ satisfies
\[
\max_{v\in V}\Bigg\{\frac{\tanh(r(v)/2)}{\tanh(\tilde r(v)/2)}\Bigg\}
\,=\,\frac{\tanh(r(v_0)/2)}{\tanh(\tilde r(v_0)/2)}\,>\,1.
\]
From Lemma \ref{L-2-11}, it follows that
    \[
  k(v_0)(r)\,>\,k(v_0)(\tilde{r}).
    \]
On the other hand, note that both $r,\tilde{r}$ give the smooth circle pattern. That means
 \[
 k(v_0)(r)\,=\,k(v_0)(\tilde{r})\,=\,0,
 \]
which leads to a contradiction.

As a result, the pairs $(\mu,\mathcal P)$ and $(\tilde \mu,\tilde{\mathcal P})$ are isometric.
\end{proof}

\begin{proposition}\label{P-4-1}
As $n\to\infty$, no disk in $\mathcal P_n$ degenerates to a point.
\end{proposition}
\begin{proof}
Assume that the proposition is untrue. Let $V_0\subset V$ be the set of vertices whose disks degenerate to points. Suppose that $CK(V_0)$ is the union of those cells of $\mathscr{D}$ having at least one vertex in $V_0$. Assume that $CK(V_0)$ is connected and is homotopic to a surface of topological type $(g_0,l_0)$, where $g_0,\,l_0$ respectively denote the genus and the number of boundary components. Lemma \ref{L-2-12} implies that each interstice in $CK(V_0)$ degenerates to a point.

If $g_0>0$, or $l_0>1$, as the disks $\{\,D_v\,\}_{v\in V_0}$ degenerate to points, there exists a simple closed geodesic $\gamma_n$ on $(S,\mu_n)$ with the length $\ell(\gamma_n)$ tending to zero. Due to the Collar theorem \cite[Chap. 4]{Bu},  $\gamma_n$ gives an embedding cylinder domain $Cy(\gamma_n)$ in $(S,\mu_n)$ such that
\[
Cy(\gamma_n)\ =\ \big\{\,p\ \in\  (S,\mu_n)\,|\,\dist(p,\gamma_n)\ \leq\  d_n\,\big\},
\]
where $d_n>0$ satisfies
\[
\sinh\big(\ell(\gamma_n)/2\big)\sinh d_n\,=\,1.
\]

As $\ell(\gamma_n)\to 0$, the cylinder becomes so long  that the diameter of $\big(S,\mu_n\big)$ tends to infinity. However, because the radius of any disk in $\mathcal P_n$ is upper bounded, the diameter of $\big(S,\mu_n\big)$ is also upper bounded. This leads to a contradiction.

In case that $g_0=0,\,l_0=1$, then $CK(V_0)$ is simply-connected. Let $e_1,\cdots, e_s $ be the edges forming the  \textbf{pseudo-Jordan} path bounding $CK(V_0)$. According to the condition, we have
\begin{equation}\label{E-p}
\sum\nolimits_{i=1}^s\Theta_n(e_i)
\ = \ \sum\nolimits_{i=1}^s\Theta(e_i)\ <\ (s-2)\pi.
\end{equation}
On the other hand, let $V_\partial$ be the set of end vertices of $e_1,\cdots, e_s$. As the disks associated to the those vertices in $V_0$ degenerate, the interstices in $CK(V_0)$ degenerates to points. Consequently, the disks associated to the vertices in $V_\partial$ will meet at a common point. It follows that
\[
\sum\nolimits_{i=1}^s\Big(\pi-\Theta_n(e_i)\Big)\to 2\pi,
\]
which contradicts to (\ref{E-p}).

Thus the proposition is proved.
\end{proof}

\subsection{Ideal circle patterns} Observe that an \textbf{ideal} circle pattern is the limiting case of these circle patterns in the above part. Let us prove Theorem \ref{T-1-7} via Theorem \ref{T-1-5}.

\begin{proof}[\textbf{Proof of Theorem \ref{T-1-7}}]
 Assume that $\Theta: E\mapsto (0,\,\pi)$ is a function satisfying the conditions \textbf{(H1)} and \textbf{(H2)}. We define a family of weight functions $\Theta_{\varepsilon}$ by setting $\Theta_\varepsilon(e)=\Theta(e)-\varepsilon$ for each $e\in E$, where $\varepsilon$ is any positive number such that $\Theta_\varepsilon(e)\in(0,\pi)$. We check that
\[
\sum\nolimits_{i=1}^s\Theta_{\varepsilon}(e_l)<(s-2)\pi,
 \]
 whenever $e_1,e_2,\cdots,e_s$ form a \textbf{pseudo-Jordan} path in $S$. Pick one $[h(G)]=(h_1,\cdots,h_{|F|})\in \mathcal T_{G}$. Owing to Theorem \ref{T-1-5}, there exists a circle pattern pair $(\mu_\varepsilon,\mathcal P_\varepsilon)$ realizing $(G,\Theta_{\varepsilon})$ whose interstices are equipped with conformal structures assigned by $[h(G)]$.

 Let $I_{1,\epsilon},\cdots,I_{|F|,\varepsilon}$ be the interstices of $\mathcal P_{\varepsilon}$. For each $\alpha\in\{1,\cdots,|F|\}$, there exists the mark-preserving conformal mapping
 \[
\varphi_{\alpha,\varepsilon}: h_{\alpha}(I^0_\alpha)\mapsto I_{\alpha,\varepsilon}.
\]
 Because of Lemma \ref{L-2-11}, the radius of each disk in $\mathcal P_\varepsilon$ is upper bounded. As a result, we can extract a subsequence $\{(\mu_{\varepsilon_k},\mathcal P_{\varepsilon_k})\}$  convergent to a pair $(\mu_\ast,\mathcal P_\ast)$. Under the condition \textbf{(H2)}, similar argument to Proposition \ref{P-4-1} implies that no disk in $\mathcal P_\ast$ degenerates. Meanwhile, due to Mentel's theorem, a convergent subsequence of $\{\varphi_{\alpha,\varepsilon_k}\}$ exists. For simplicity, we still denote it by $\{\varphi_{\alpha,\varepsilon_k}\}$.

To show that every interstice of $\mathcal P_\ast$ consists of a point,
we claim that
\[
\varphi_{\alpha,\varepsilon_k}\,\to\, \text{const}.
\]
Otherwise, Rouch\'{e}'s theorem implies that
\[
\varphi_{\alpha,\varepsilon_k}\,\to\, \phi_{\alpha},
\]
where $\phi_{\alpha}$ is a mark-preserving conformal map. Then the boundary of the $\alpha$-th interstice of $\mathcal P_\ast$ is non-empty and forms a curved polygon, which is called the interstice polygon. Replacing each side of the interstice polygon with geodesic segment, we obtain a hyperbolic polygon, called the rectified polygon.
See FIGURE 5.

\begin{figure}[htbp]\centering
\includegraphics[width=0.95\textwidth]{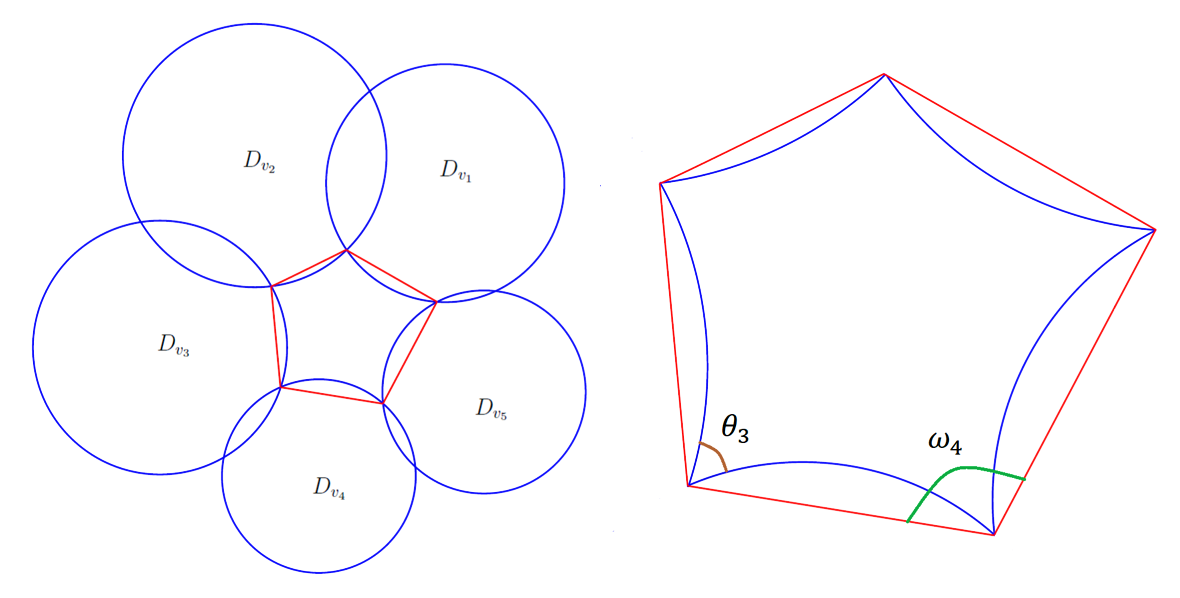}
\caption{Interstice polygon and rectified polygon}
\end{figure}

Let $D_{v_1},\cdots D_{v_m}$, in counter-clockwise order, be all the  disks adjacent to the above interstice. For $i=1,\cdots,m$, assume that $\theta_i$ is the $i$-th inner angle of the interstice polygon. Here the subscript $i$ means that the angle is at the endpoint in $D_{v_i}\cap D_{v_{i+1}}$ (set $v_{m+1}=v_1$). Define the $i$-th inner angle $w_i$ of the rectified polygon similarly. Observe that
\[
\theta_i<w_i
\]
and
\[
\theta_i\,=\,\lim_{k\to\infty}\big(\,\Theta(e_i)-\varepsilon_k\,\big)
\,=\,\Theta(e_i),
\]
where $e_i\in E$ is the edge between $v_i$ and $v_{i+1}$.  Combining with the Gauss-Bonnet formula, we derive that
\[
\sum\nolimits_{i=1}^m\Theta(e_i)\,=\,\sum\nolimits_{i=1}^m\theta_i\ <\sum\nolimits_{i=1}^m \omega_i\,<\, (m-2)\pi.
\]
On the other hand, according to condition \textbf{(H1)},
\[
\sum\nolimits_{i=1}^m\Theta(e_i)\,=\,(m-2)\pi.
\]
This leads to a contradiction.

Setting $(\mu,\mathcal P)=(\mu_\ast,\mathcal P_\ast)$,  we check that it is an \textbf{ideal } circle pattern pair with \textbf{primitive contact graph} $G$ and the exterior intersection angles given by $\Theta$.

For the rigidity part, suppose that $(\mu,\mathcal P)$ and $(\tilde{\mu}, \tilde{\mathcal P})$ are two ideal circle pattern pairs realizing $(G,\Theta)$. Owing to Lemma \ref{L-2-13}, similar argument to proof of Theorem \ref{T-1-5} implies that $(\mu,\mathcal P)$ and $(\tilde{\mu}, \tilde{\mathcal P})$ are isometric.
\end{proof}

\subsection{Packable surfaces}In this part some knowledge on Teichm\"{u}ller theory is needed. The readers can refer to  \cite{Ahl,Ima} for detailed backgrounds. Let $T(S)$ be the Teichm\"{u}ller space of the compact oriented surface $S$ of genus $g>1$. Then $T(S)$ parameterizes the equivalence classes of marked hyperbolic metrics on $S$. Here a marking is an isotopy class of orientation-preserving homeomorphism from $S$ to itself. And two marked hyperbolic metrics $\mu,\tilde{\mu}\in \mathcal T(S)$ are equivalent if there exists an isometry $\phi: (S,\mu)\mapsto(S,\tilde{\mu})$ isotopic to the identity map.

\begin{proof}[\textbf{Proof of Theorem \ref{T-1-7}}]
For any $\mu\in T(S)$, by adding proper disks into $(S,\mu)$, it is not hard to see that  $(S,\mu)$  supports a circle packing $\mathcal P$ whose contact graph is isomorphic to the 1-skeleton of a cellular decomposition of $S$. Let $I_1,\cdots,I_{|F|}$ denote all the interstices of $\mathcal P$. For each $\alpha$, lifting the inclusion $\imath_\alpha:I_\alpha\mapsto (S,\mu)$ to the universal cover $\mathbb D$ gives a map $h_\alpha: I_\alpha\mapsto \mathbb D$.

The following process is similar to the proof of Theorem \ref{T-1-5}. Lay down the regular hexagonal circle packing  in $\mathbb{C}$, with each circle of radius $1/n$. By using the boundary component $h_\alpha(I_\alpha)$ like a cookie-cutter, we obtain a circle packing $\mathcal Q_{n,\alpha}$ which consists of all disks intersecting with the region $h_\alpha(I_\alpha^0)$.

For $\alpha=1,2,\cdots,|F|$,  jointing the contact graph $K_{n,\alpha}$ of $\mathcal {Q}_{n,\alpha}$ to the original graph $G$ along the corresponding
boundaries, in this way we obtain a triangular graph. By Theorem \ref{T-1-1}, similar argument to the proof of Theorem \ref{T-1-5} shows that there exists a hyperbolic metric $\mu_n$ on $S$ such that $(S,\mu_n)$ supports a circle packing $\mathcal P_n$ realizing $G$.

For $\{(\mu_n,\mathcal P_n)\}$, let us extract a subsequence $\{(\mu_{n_k},\mathcal P_{n_k})\}$ convergent to  $(\mu_\infty,\mathcal P_\infty)$. Due to Proposition \ref{P-4-1},  $(\mu_\infty,\mathcal P_\infty)$ is a non-degenerating  pair with contact graph $G$.
Moreover,  Rodin-Sullivan's consequence \cite{RS} implies that the interstices of $\mathcal P_\infty$ and $\mathcal P$ share the same conformal structures. By the rigidity part of Theorem \ref{T-1-5}, we have $\mu_\infty=\mu$.

In summary, there exists a sequence $\{\mu_{n_k}\}$ convergent to $\mu$ such that each $(S,\mu_{n_k})$ is packable. Thus the theorem is proved.
\end{proof}

\section{Appendix 1: several results on univalent functions}
In this appendix, we collect some results on univalent functions.

Let $f$ be a holomorphic function in $\mathbb{D}$ such that $|f(z)|< 1$. Set
\[
M(f ; z):\ =\frac{|f'(z)|(1-|z|^2)}{1-|f(z)|^2}.
\]

Here is a list of properties of this quantity.
\begin{itemize}
\item[$(i)$] For any $\beta$ in the isometry group  $Aut(\mathbb{D})$ of  $\mathbb{D}$,
\begin{equation}\label{E-12}
M(\beta ; z)\ \equiv \ 1;
\end{equation}
\item[$(ii)$] For any $\beta,\alpha\in Aut(\mathbb{D})$,
\begin{equation}\label{E-13}
M(\beta\circ f\circ\alpha ; z)\ =\  M(f ; \alpha(z));
\end{equation}
\item[$(iii)$] Suppose that $f$  is invertible. Denote $w = f(z)$. Then
\begin{equation}\label{E-14}
    M(f ; z)M(f^{-1} ; w)\ \equiv \ 1.
\end{equation}
\end{itemize}

In addition, the following two formulas will be helpful.
\begin{equation}\label{E-15}
 \Delta\log\frac{1}{1-|z|^2}\ =\ \bigg(\frac{2}{1-|z|^2}\bigg)^2.
 \end{equation}
 \begin{equation}\label{E-16}
 \Delta\log\frac{|f'(z)|}{1-|f(z)|^2}\ =\ \bigg(\frac{2|f'(z)|}{1-|f(z)|^2}\bigg)^2.
\end{equation}

The lemma below was due to  Schiffer-Hawley \cite{S-N}. It plays an important role in He-Liu's work \cite{HL}.
\begin{lemma}\label{L-9-1}
Suppose $\Omega\,(resp.\,\widetilde{\Omega})\subset \mathbb{C}$ is a domain bounded by $m$ close analytic curves $C_\tau\,(\text{resp}.\,\tilde{C}_\tau)$ $(\tau=1,\cdots,m)$, now $w=f(z)$ is a univalent conformal mapping from $\Omega$ to $\widetilde{\Omega}$ and carries the curve system $C_\tau$ to corresponding curve system curves $\tilde{C}_\tau$. We assume $f(z)$
is analytic in the closed domain $\Omega \cup\partial \Omega$, where $\partial \Omega=\cup C_\tau$. Suppose that $z=z(s)$ and $w=w(\sigma)$ are parametric representation in terms of their arc length, then we have the following formula:
$$\frac{\partial}{\partial n}\log{|f'(z)|}\ =\ k(s)-\tilde k(\sigma)|f'(z)|$$
where the operator $\partial/\partial n$ denotes differential with respect to the exterior normal on the boundary curves $C=\cup C_\tau$, and $k(s),\tilde k(\sigma)$ are curvatures of curve $C=\cup C_\tau$ and $\tilde{C}=\cup \tilde{C}_\tau$
in corresponding parameter point.
\end{lemma}

\section{Appendix 2: Rigidity lemma}
 Recall that the interstices $I,\tilde{I}$ of two $m$-circle configurations $\mathcal P,\tilde{\mathcal P}$ are equipped with the same conformal structure if and only if there exists a mark-preserving conformal mapping $\phi$ between them.

\begin{proof}[\textbf{Proof of Lemma \ref{L-2-9}}]
We consider the function $H(\phi ; z):=\log M(\phi;z)$.
Suppose it attains minimum at $z_0$. There are four cases to distinguish.

$(\textrm{I})$ If $z_0$ belongs to the interior of $I$, we have
 \[
 \Delta H(\phi ; z_0)\ \geq \ 0.
 \]
 Note that
 \[
H(\phi ; z_0)\,=\,\log\frac{|\phi'(z_0)|(1-|z_0|^2)}{1-|\phi(z_0)|^2}\,=\,\log\frac{|\phi'(z_0)|}{1-|\phi(z_0)|^2}-\log\frac{1}{1-|z_0|^2}.
\]
 From (\ref{E-15}) and (\ref{E-16}), it follows that
\[
\bigg(\frac{2|\phi'(z_0)|}{1-|\phi(z_0)|^2}\bigg)^2-\bigg(\frac{2}{1-|z_0|^2}\bigg)^2\ \geq \ 0,
\]
which implies
\[
 M(\phi;z_0)\ \geq \ 1.
\]
Since $z_0$ is the minimal point of $H(\phi ; z)$,
\[
 H(\phi;z)\ \geq \, H(\phi ; z_0)\ =\ \log M(\phi;z_0)\ \geq\ 0.
\]
As a result,
\[
 M(\phi;z)\ \geq \ 1.
\]

$(\textrm{II}).$ If $z_0$ is in the boundary of $I$, we assume $z_0$ locates in the $l$-th circle $C_l$. By compositing proper hyperbolic isometries $\beta,\alpha\in Aut(\mathbb{D})$, we may assume that $C_l,\tilde C_l$ are the same circle centering at the origin with Euclidean radius $\displaystyle\mathbbm{r}_l=\tanh(r_l/2)=\tanh(\tilde r_l/2)$. Owing to (\ref{E-13}),  for simplicity, we still use  $H(\phi;\cdot)$ instead of $H(\beta\circ\phi\circ\alpha;\cdot)$. Since  $z_0$ and $\phi(z_0)$ belong to $C_l$ and $\tilde C_l$, respectively, we have
\begin{equation}\label{E-17a}
|z_0|\,=\,|\phi(z_0)|\,=\,\mathbbm{r}_l.
\end{equation}
Because $z_0$ is the minimal point, we obtain
\begin{equation}\label{E-17}
\frac{\partial}{\partial n_l}H(\phi;z_0)\ \leq \ 0.
\end{equation}
It follows from Lemma \ref{L-9-1} that
\begin{equation}\label{E-18}
\frac{\partial}{\partial n_l}\log{|\phi'(z_0)|}\ =\ \frac{1}{\mathbbm{r}_l} \ -\frac{1}{\mathbbm{r}_l}|\phi'(z_0)|.
\end{equation}
A direct calculation gives
\begin{equation}\label{E-19}
\frac{\partial}{\partial n_l}\log(1-|z_0|^2)\ = \ \frac{2|z_0|}{1-|z_0|^2},
\end{equation}
\begin{equation}\label{E-20}
\Big|\frac{\partial}{\partial n_l}\log(1-|\phi(z_0)|^2)\Big|\ \leq \ \frac{2|\phi(z_0)||\phi'(z_0)|}{1-|\phi(z_0)|^2}.
\end{equation}
Note that
\[
\frac{\partial}{\partial n_l}H(\phi;z_0)\,=\,\frac{\partial}{\partial n_l}\log{|\phi'(z_0)|}+\frac{\partial}{\partial n_l}\log(1-|z_0|^2)-\frac{\partial}{\partial n_l}\log(1-|\phi(z_0)|^2).
\]
From (\ref{E-17a}), (\ref{E-17}),(\ref{E-18}),(\ref{E-19}),(\ref{E-20}), we deduce that
\[
|\phi'(z_0)|\ \geq \ 1.
\]
Consequently,
\[
 M(\phi;z)\, \geq\, M(\phi;z_0)\,\geq\, 1.
\]

$(\textrm{III}).$ If $z_0$ is an intersection point of two adjacent circles $C_j,C_{j+1}$ meeting in angle $\Theta_j \neq0$. First, we extend the conformal map $\phi$ by reflections across $C_j$ and $C_{j+1}$.  Repeat this procedure until the exterior normal vector $\textbf{n}_j$ and
 $\textbf{n}_{j+1}$ of $C_j$ and $C_{j+1}$ at $z_0$ belong to the domain of the extended conformal mapping. For simplicity, we still denote it by $\phi$. We have
 \[
 \frac{\partial}{\partial \vec{v}}H(\phi;z_0)\ \leq 0,
 \]
 where $\vec{v}$ denotes the unit vector bisecting the angle between $\textbf{n}_j$ and $\textbf{n}_{j+1}$. A simple computation gives
 \[
 \frac{\partial}{\partial \vec{v}}H(\phi;z_0)\,=\, \frac{1}{2\sin(\Theta_j/2)}\bigg(\,\frac{\partial}{\partial n_j}H(\phi;z_0)+\frac{\partial}{\partial n_{j+1}}H(\phi;z_0)\,\bigg).
 \]
That means
\[
\frac{\partial}{\partial n_j}H(\phi;z_0)\leq 0,
\]
or
\[
\frac{\partial}{\partial n_{j+1}}H(\phi;z_0)\leq 0.
\]
We transform this case into $(\textrm{II})$.

$(\textrm{IV}).$ If $z_0$ is an intersection point of two adjacent circles $C_j,C_{j+1}$ meeting in angle $\Theta_j =0$. By compositing  proper hyperbolic isometries, we assume that $z_0=\phi(z_0)=0$. Then $C_j, C_{j+1}$ (resp. $\tilde{C}_j,\tilde{C}_{j+1}$) tangent at the origin. Thus
 \[
 M(\phi,z_0)=|\phi'(0)|.
 \]
 According to the definition,
 \[
 \phi'(0)=\lim_{z\to 0}\phi(z)/z=\lim_{z\to 0}\frac{1/z}{1/\phi(z)}=\lim_{w\to\infty}w/g(w),
 \]
 where $w=1/z$ and $g(w)=1/\phi(1/w)$.

After the map $w=1/z$, the circles $C_j, C_{j+1}$  become a pair of parallel horizontal lines with width (in Euclidean sense)
 \[
 \frac{1}{2}\Bigg(\,\frac{1}{\tanh (r_j/2)}+\frac{1}{\tanh(r_{j+1}/2)}\,\Bigg).
 \]
 Similarly, $\tilde{C}_j,\tilde{C}_{j+1}$ become a pair of parallel lines with width
 \[
 \frac{1}{2}\Bigg(\,\frac{1}{\tanh(\tilde r_j/2)}+\frac{1}{\tanh(\tilde r_{j+1}/2)}\,\Bigg).
 \]
Because $r_j=\tilde r_j$ and $r_{j+1}=\tilde r_{j+1}$,  a routine calculation gives
 \[
 \lim_{w\to\infty}w/g(w)=\bigg(\frac{1}{\tanh (r_j/2)}+\frac{1}{\tanh(r_{j+1}/2)}\bigg)\bigg/\bigg(\frac{1}{\tanh(\tilde r_j/2)}+\frac{1}{\tanh(\tilde r_{j+1}/2)}\bigg)=1,
 \]
which implies that
 \[
 M(\phi;z)\ \geq \ 1.
 \]

To summarise, we always have
\[
M(\phi;z)\ \geq \ 1, \forall\, z\in I.
\]
Similarly,
 \[
 M(\phi^{-1};w)\ \geq \ 1, \forall\, w\in \tilde{I}.
 \]
 From (\ref{E-14}), it follows that
 \[
 M(\phi;z)\ = \ 1, \forall\; z\in I.
 \]

Let  $\displaystyle ds$ and $d\tilde{s}$ be the hyperbolic metrics on $I,\tilde {I}$, respectively. The above fact implies that $\displaystyle \phi^{\star}d\tilde{s}=ds$. We thus finish the proof.
 \end{proof}

\section{Appendix 3: Maximal principle lemma}
The main purpose of this appendix is to show Lemma \ref{L-2-10}.

\begin{proof}[\textbf{Proof of Lemma \ref{L-2-10}}]
There is a mark-preserving conformal mapping $\phi: I\mapsto \tilde{I}$ between the interstices of $\mathcal P$ and $\mathcal P$. As illustrated in FIGURE 6, we divide $\vartheta_i$ into three parts $\vartheta_{i0},\vartheta_{i1},\vartheta_{i2}$. Namely,
\[
\vartheta_i\, =\, \vartheta_{i0}+\vartheta_{i1}+\vartheta_{i2}.
\]
Similarly, set
\[
\tilde\vartheta_i \,=\, \tilde\vartheta_{i0}+\tilde\vartheta_{i1}+\tilde\vartheta_{i2}.
\]
Note that
\[
\max_{1\leq \tau\leq m}\Bigg\{\frac{\tanh(r_\tau/2)}{\tanh(\tilde r_\tau/2)}\Bigg\}
\,=\,\frac{\tanh(r_i/2)}{\tanh(\tilde r_i/2)}\,>\,1.
\]
By Lemma \ref{L-2-7}, it is not hard to see
\begin{equation}\label{E-21}
\vartheta_{i1}\,<\,\tilde\vartheta_{i1},\quad \vartheta_{i2}\,<\,\tilde\vartheta_{i2}.
\end{equation}

\begin{figure}[htbp]\centering
\includegraphics[width=0.7\textwidth]{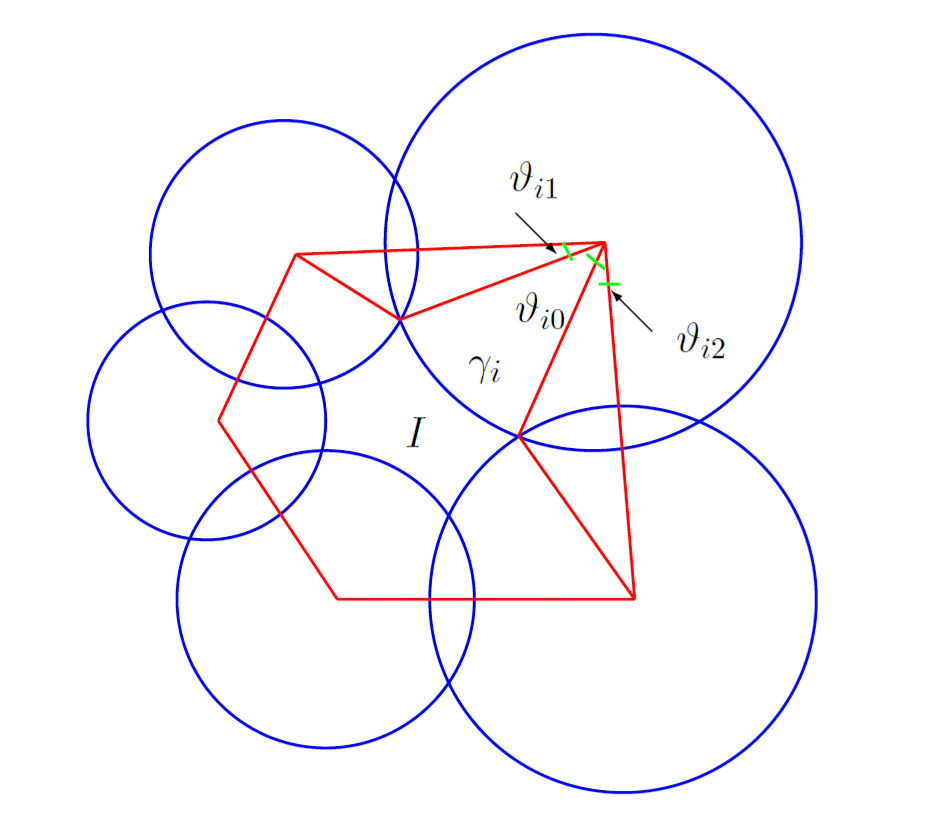}
\caption{}
\end{figure}

It suffice to show that $\vartheta_{i0}<\tilde\vartheta_{i0}$. To achieve the goal, let us consider $M(\phi;z)$. Suppose $M(\phi;z)$ attains its minimum at a point. Using similar argument to the proof of Lemma \ref{L-2-9}, we obtain
\begin{equation}\label{E-22}
M(\phi;z)\ \geq \ \frac{\tanh(\tilde r_i/2)}{\tanh(r_i/2)}.
\end{equation}
Actually, in case $(\textrm{I})$, we have
\[
M(\phi;z)\,\geq\,1.
\]
In case $(\textrm{II})$ or $(\textrm{III})$, we derive
\[
M(\phi;z)\,\geq\, \min_{1\leq\tau\leq m}\Bigg\{\frac{\tanh( \tilde{r}_\tau/2)}{\tanh( r_\tau/2)}\Bigg\}\,=\,\frac{\tanh(\tilde r_i/2)}{\tanh(r_i/2)}.
\]
In case $(\textrm{IV})$, we deduce
\[
M(\phi;z)\,\geq\,\bigg(\frac{1}{\tanh (r_j/2)}+\frac{1}{\tanh(r_{j+1}/2)}\bigg)\bigg/\bigg(\frac{1}{\tanh(\tilde r_j/2)}+\frac{1}{\tanh(\tilde r_{j+1}/2)}\bigg),
\]
which also implies (\ref{E-22}).

Let us integrate $ds$ (resp. $d\tilde{s}$) along the $i$-th side $\gamma_i$ (resp. $\tilde \gamma_i$) boundary of $I$ (resp. $\tilde{I}$). Note that $\displaystyle \phi^{\star} d\tilde{s}\ = \ M(\phi;z)ds$. Due to (\ref{E-22}), we obtain
\begin{equation}\label{E-23}
\int_{\tilde\gamma_i}\ \mathrm{d}\tilde{s}\ \geq \ \frac{\tanh(\tilde r_i/2)}{\tanh(r_i/2)}\int_{\gamma_i}\ \mathrm{d}s.
\end{equation}
Moreover, a direct calculation gives
\[
\sinh r_i \ \vartheta_{i0}\ =\ \int_{\gamma_i}\ \mathrm{d}s,
\]
and
\[
\sinh\tilde r_i \ \tilde\vartheta_{i0}\ =\ \int_{\tilde\gamma_i}\ \mathrm{d}\tilde{s}.
\]
It follows that
\begin{equation}\label{E-24}
\vartheta_{i0}\,\leq\,\frac{\cosh^2(\tilde{r}_i/2)}{\cosh^2(r_i/2)}\tilde{\vartheta}_{i0}\,<\,\tilde{\vartheta}_{i0}.
\end{equation}
Combining  (\ref{E-21}) and (\ref{E-24}), we show that
\begin{equation}\label{E-25}
\vartheta_i \, <\,\tilde\vartheta_i.
\end{equation}
Thus the lemma is proved.
\end{proof}

\begin{proof}[\textbf{Proof of Lemma \ref{L-2-11}}]
  Following the above method, we write $\vartheta_i$ as $\vartheta_i = \vartheta_{i0}+\vartheta_{i1}+\vartheta_{i2}$. By Lemma \ref{L-2-8}, as $r_i\to+\infty$, we have
   \[
  \vartheta_{i1}\to 0,\quad \vartheta_{i2}\to 0.
  \]
  Moreover, it follows from (\ref{E-24}) that
  \[
  \vartheta_{i0}\to 0
  \]
   as $r_i\to +\infty$. We finish the proof.
\end{proof}

\begin{proof}[\textbf{Proof of Lemma \ref{L-2-13}}]
 Observe that an ideal multi-circle configuration can be divided into parts, where each part is a two-circle configuration (see FIGURE 7). Thus the lemma follows from Lemma \ref{L-2-5}, Lemma \ref{L-2-7} and Lemma \ref{L-2-8}.
\end{proof}

\begin{figure}[htbp]\centering
\includegraphics[width=0.6\textwidth]{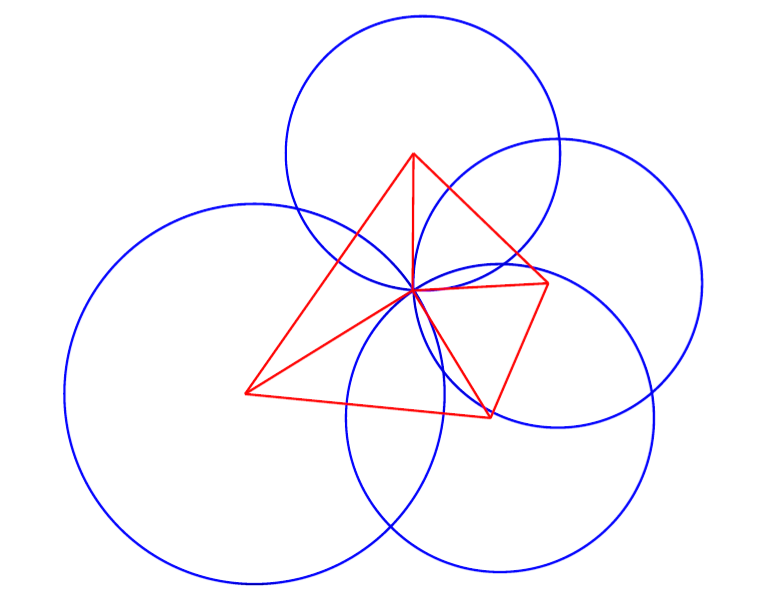}
\caption{An ideal 4-circle configuration}
\end{figure}
\section{Acknowledgement}

The author  would like to thank J. Liu for encouragement and helpful comments. He also thanks to L. Liu  for the help of  FIGURE 6. Part of this work was done when the author was visiting the Department of Mathematics, Rutgers University, he would like to thank for its hospitality. He also would like to thank the China Scholarship Council for financial support (File No. 201706135016).


\begin{thebibliography}{99}
\bibitem{Ahl}L V Ahlfors, Lectures on quasiconformal mappings, AMS \textbf{10}, New York, 1966.
\bibitem{Boben}A.I. Bobenko \& B. A. Springborn, Variational principles for circle patterns and Koebe's theorem, Trans. Amer. Math. Soc. \textbf{356} (2004) 659-689.
\bibitem{BS3} P.L. Bowers \& K. Stephenson, The set of circle packing points in the Teichm\"{u}ller space of a surface of finite conformal type is dense, Math. Proc. Cambridge Philos. Soc. \textbf{111} (1992) 487-513.
\bibitem{BK1}R. Brooks, On the deformation theory of classical Schottky groups, Duke Math. J. \textbf{52} (1985) 1009-1024.
\bibitem{BK2}R. Brooks, Circle packings and co-compact extensions of Kleinian groups. Invent. Math. \textbf{86} (1986) 461-469.
\bibitem{Bu}P. Buser, Geometry and spectra of compact Riemann surfaces, Progress in Mathematics, \textbf{106}, Springer, Birkh\"{a}user, Boston, 1992.
\bibitem{Luo1}B. Chow \& F. Luo, Combinatorial Ricci flows on surfaces, J. Differential Geom. \textbf{63} (2003) 97-129.
\bibitem{Ge-Jiang} H. Ge \& W. Jiang, On the deformation of inversive distance circle packings, II, J. Funct. Anal. \textbf{272(9)} (2017) 3573-3595.
\bibitem{Ge-Xu}H. Ge \& X. Xu, A discrete Ricci flow on surfaces with hyperbolic background geometry, Int. Math. Res. Not. IMRN. \textbf{11} (2017) 3510-3527.
    \bibitem{Guo}R. Guo, Local rigidity of inversive distance circle packing,Trans. Amer. Math. Soc. \textbf{363} (2011) 4757-4776.
\bibitem{HL}Z. He \& J. Liu, On the Teichm\"{u}ller theory of circle patterns, Trans. Amer. Math. Soc. \textbf{365} (2013) 6517-6541.
\bibitem{Hu-L}X. Huang \& J. Liu, Characterizations of circle patterns and finite convex polyhedra in hyperbolic 3-space, Math. Ann. \textbf{368(1-2)} (2017) 213-231.
\bibitem{Ima}Y. Imayoshi \& M. Taniguchi, An introduction to Teichm\"{u}ller spaces, Springer-Verlag, Tokyo, 1992.
\bibitem{LV} O. Lehto \& K.I. Virtanen, Quasiconformal mappings in the plane, Springer-Verlag, New York, 1973.
\bibitem{L-Z}J. Liu \& Z. Zhou, How many cages midscribe an egg, Invent.Math. \textbf{203} (2016) 655-673.
\bibitem{Rivin1}I. Rivin, A characterization of ideal polyhedra in hyperbolic 3-space,  Ann. of Math. \textbf{143(1)} (1996) 51-70.
\bibitem{RS}B. Rodin \& D. Sullivan, The convergence of circle packings to the Riemann mapping, J. Differential Geom. \textbf{26(2)} (1987) 349-360.
\bibitem{Schl}J.M. Schlenker, Hyperideal circle patterns, Math. Res. Lett. \textbf{12(1)} (2005) 85-112.
\bibitem{S-N}M. Schiffer \& N.S. Hawley,  Connections and conformal mapping, Acta Math. \textbf{107(3-4)} (1962) 175-274.
\bibitem{Stephenson1}K. Stephenson, Introduction to circle packing: The theory of discrete analytic functions, Cambridge University Press,  Cambridge, 2005.

\bibitem{Thurston}W.P. Thurston,The geometry and topology of three-manifolds, Princeton lecture notes, 1979.
\bibitem{Xu} X. Xu, Rigidity of inversive distance circle packings revisited, arXiv:1705.02714v3.
\bibitem{Zhou}Z. Zhou, Circle pattern with obtuse exterior intersection angles, arXiv:1703.01768v2.
\end{thebibliography}
\end{document}